\documentclass{amsart}
\usepackage[active]{srcltx} 
\usepackage{graphicx,ytableau,hyperref,xcolor}
\newtheorem{thm}{Theorem}[section]
\newtheorem{lem}[thm]{Lemma}
\theoremstyle{remark}
\newtheorem{rem}[thm]{Remark}
\numberwithin{equation}{section}
\newcommand{\Cov}{\textup{Cov\,}}
\newcommand{\EE}{\textup{I}\!\textup{E\,}}
\newcommand{\Oh}{\mathcal{O}}
\newcommand{\PP}{\textup{I}\!\textup{P\,}}
\newcommand{\Var}{\textup{Var\,}}
\renewcommand{\(}{\left(}
\renewcommand{\)}{\right)}
\renewcommand{\i}{\mathrm{i}}
\newcommand{\eps}{\varepsilon}
\begin{document}
\title[Asymptotics of some Plancherel averages via polynomiality results]
{Asymptotics of some Plancherel averages via polynomiality results}

\author{Werner Schachinger}
\address{Dept.\ of Statistics and Operations Research (ISOR), University of Vienna\\
Vienna, Austria}
\email{Werner.Schachinger@univie.ac.at}
\thanks{}

\date{\today}
\begin{abstract}
Consider Young diagrams of $n$ boxes distributed according to the Plancherel measure. So those diagrams could be the output of the RSK algorithm, when applied to random permutations of the set $\{1,\ldots,n\}$. Here we are interested in asymptotics, as $n\to \infty$, of expectations of certain
functions of random Young diagrams, such as the number of bumping steps of the RSK algorithm that leads to that diagram, the side length of its Durfee square, or the logarithm of its probability. We can express these functions in terms of hook lengths or contents of the boxes of the diagram,
which opens the door for application of known polynomiality results for Plancherel averages. We thus obtain representations of expectations
as binomial convolutions, that can be further analyzed with the help of
Rice's integral or Poisson generating functions. Among our results is
a very explicit expression for the constant appearing in the almost equipartition property of the Plancherel measure.

\end{abstract}
\maketitle
\section{Introduction}
We identify
Young diagrams (sets consisting of left aligned decreasingly ordered rows of square boxes)
with partitions $\lambda=(\lambda_1,\ldots,\lambda_k)$ with
$\lambda_1\ge\lambda_2\ge\cdots\ge\lambda_k$, and denote $|\lambda|=\lambda_1+\ldots+\lambda_k$. The notation $\lambda\vdash n$ then signifies that $\lambda$ is a partition of $n$, i.e., $|\lambda|=n$. We let
$Y(\pi)=Y_\lambda:=\sum_{\ell=1}^k\lambda_\ell(\ell-1)$ denote the number
of bumping steps of the Robinson--Schensted algorithm (see Figures~\ref{fig:RobSchen} and \ref{fig:BumpSteps})
when applied to a permutation $\pi$
that is mapped to a pair of standard Young tableaux of shape $\lambda$.
A standard Young tableau is a Young diagram $\lambda$ filled with numbers $1,\ldots,|\lambda|$ in
a way such that numbers in each row and each column are increasing.
See e.g.~\cite[sec.\,1.6]{Rom15} or \cite[sec.\,3.1]{Sag01} for nice expositions of the algorithm and references to the original articles
by Gilbert~de~Beauregard~Robinson, by Craige~Eugene~Schensted, and by Donald~Ervin~Knuth, who significantly widened the scope of the algorithm, the abbreviation with reference to all three authors, \emph{RSK algorithm}, now frequently being used also to refer to the original Robinson--Schensted algorithm.
\begin{figure}
\centering
\ytableausetup{centertableaux}\small\footnotesize
\begin{tabular}{c|cccccccc}
\normalsize&7&5&1&8&6&3&4&2\\ \hline\\[-.3cm]
\tiny{\normalsize$P$}&\begin{ytableau}
7
\end{ytableau}&
\begin{ytableau}
5\\
7
\end{ytableau}&
\begin{ytableau}
1 \\
5 \\
7
\end{ytableau}&
\begin{ytableau}
1 & 8 \\
5 \\
7
\end{ytableau}&
\begin{ytableau}
1 & 6 \\
5 & 8 \\
7
\end{ytableau}&
\begin{ytableau}
1 & 3 \\
5 & 6 \\
7 & 8
\end{ytableau}&
\begin{ytableau}
1 & 3 & 4\\
5 & 6 \\
7 & 8
\end{ytableau}&
\begin{ytableau}
1 & 2 & 4\\
3 & 6 \\
5 & 8 \\
7
\end{ytableau}\\[0.9cm]
{\normalsize$Q$}&\begin{ytableau}
1
\end{ytableau}&
\begin{ytableau}
1 \\
2
\end{ytableau}&
\begin{ytableau}
1 \\
2 \\
3
\end{ytableau}&
\begin{ytableau}
1 & 4 \\
2 \\
3
\end{ytableau}&
\begin{ytableau}
1 & 4 \\
2 & 5\\
3
\end{ytableau}&
\begin{ytableau}
1 & 4 \\
2 & 5\\
3 &6
\end{ytableau}&
\begin{ytableau}
1 & 4 &7\\
2 & 5\\
3 &6
\end{ytableau}&
\begin{ytableau}
1 & 4 &7\\
2 & 5\\
3 &6\\
8
\end{ytableau}
\end{tabular}
\caption{Tableaux $P$ and $Q$, as they evolve when subjecting the RSK algorithm to the
permutation $\pi=(75186342)$. $P$ is constructed by row insertions of elements of $\pi$ one by one, while $Q$ is recording the position of boxes as they are added.
}
\label{fig:RobSchen}
\end{figure}

\begin{figure}
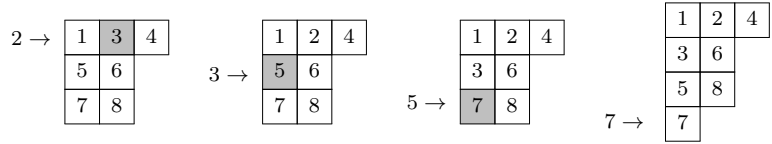

\centering%
\ytableausetup{centertableaux}\small\footnotesize
\begin{tabular}{ccccccccc}
\begin{tabular}{c}
{$2\to$}\\[.5cm]
 \\
\end{tabular}\begin{ytableau}
1 & *(lightgray)3 & 4\\
5 & 6 \\
7 & 8
\end{ytableau}&
\begin{tabular}{c}
\\[0cm]
 {$3\to$}\\[.25cm]
\end{tabular}\begin{ytableau}
1 &2 & 4\\
*(lightgray)5 & 6 \\
7 & 8
\end{ytableau}&
\begin{tabular}{c}
\\[.5cm]
 {$5\to$}\\
\end{tabular}\begin{ytableau}
1 &2 & 4\\
3 & 6 \\
*(lightgray)7 & 8
\end{ytableau}&
\begin{tabular}{c}
\\[1cm]
 {$7\to$}
\end{tabular}
\begin{ytableau}
1 & 2 & 4\\
3 & 6 \\
5 & 8 \\
7
\end{ytableau}
\end{tabular}
\caption{Detailed insertion of 2 into the second to last tableau $P$. Inserting 2, 3, and 5 into their (shaded) destination boxes causes bumps of elements down one row, being now in need of insertion themselves. In the last step, 7 is the largest element of row four, so this insertion happens without a pump.
}
\label{fig:BumpSteps}
\end{figure}
We denote by $Y_n$ the restriction of $Y(\pi)$ to permutations of the set
$\{1,2,\ldots,n\}$ chosen uniformly at random.
The Young diagrams $\lambda$ obtained by the RSK algorithm are then distributed
according to the $n$th Plancherel measure, i.e.,
$\PP\!l^{(n)}(\lambda)=\frac{f_\lambda^2}{n!}=\frac{n!}{p_\lambda^2}$,
where $f_\lambda$ is the number of standard Young tableaux of shape $\lambda$, satisfying $f_\lambda=\frac{n!}{p_\lambda}$, and where
$p_\lambda:=\prod_{u\in\lambda}h_u$ denotes the product
of the hook lengths of the diagram $\lambda$, see~\cite{FRT54}. Here the hook length $h_u$ of a particular box $u$ of $\lambda$ is one more than the number of boxes to the right of $u$ plus the number of boxes below $u$.
Note that $Y_\lambda$ has also the meaning of $|\lambda|$ times the
$y$-coordinate of the barycenter of the set
$$S_\lambda:=\{(i,j)\in\mathbb{Z}^2:0\le j\le k-1,0\le i\le \lambda_{j+1}-1\},$$
which is just the set of lower left corners of the boxes of $\lambda$ in French notation, which addresses boxes by Cartesian coordinates of the first quadrant.
Apart from here we always stick to English notation with its matrix style indexing of boxes.
Note that $X_\lambda$, the $x$-coordinate of the barycenter of $S_\lambda$, is given by
$Y_{\lambda'}$, where $\lambda'$ is the partition conjugate to $\lambda$, its parts being defined by $\lambda'_j:=|\{i:\lambda_i\ge j\}|$. Stated differently, $\lambda$ and $\lambda'$ are mirror images of one another with respect to the main diagonal (upper left to lower right). The sets of hook lengths are therefore the same for $\lambda$ and $\lambda'$, which yields invariance of Plancherel measure under conjugation. Thus $X_n$ and $Y_n$ are identical in distribution.
This allows for a representation of $\EE Y_n$ and $\Var Y_n$ in terms of
$X_n+Y_n$ and $X_n-Y_n$,
\begin{align*}
\EE Y_n&=\tfrac12\EE(X_n+Y_n),\\
\Var Y_n&=\tfrac{1}{2}\(\Var X_n+\Var Y_n\)=\tfrac{1}{4}\big(\Var (X_n+Y_n)+\Var (X_n-Y_n)\big).\end{align*}
Note that we can express $X_\lambda-Y_\lambda$, resp.\ $X_\lambda+Y_\lambda$, in terms of contents $\{c_u:u\in\lambda\}$, resp.\ hook lengths $\{h_u:u\in\lambda\}$, of the diagram $\lambda$:
\begin{equation}\label{XpmY}X_\lambda-Y_\lambda=\sum_{u\in\lambda}c_u,\qquad X_\lambda+Y_\lambda=\sum_{u\in\lambda}h_u-|\lambda|.\end{equation}
Here the content $c_u$ of a box $u=(i,j)$ of $\lambda$ is $j-i$, i.e., the column number of $u$ minus the row number of $u$, see Figure~\ref{fig:hooklengthsandcontents}
for an illustration of hook lengths, contents, and bumping step counts.
For a proof of \eqref{XpmY} note $$Y_\lambda=\sum_{i=1}^k\lambda_i(i-1)=\sum_{(i,j)\in\lambda}(i-1)=\sum_{j=1}^{\lambda_1}\sum_{i=1}^{\lambda'_j}(i-1)=\sum_{j=1}^{\lambda_1}\sum_{i=1}^{\lambda'_j}(\lambda'_j-i)=\sum_{(i,j)\in\lambda}(\lambda'_j-i),$$
and similarly
$$X_\lambda=Y_{\lambda'}=\sum_{(i,j)\in\lambda}(j-1)=\sum_{(i,j)\in\lambda}(\lambda_i-j),$$
leading to $X_\lambda-Y_\lambda=Y_{\lambda'}-Y_\lambda=\sum\limits_{(i,j)\in\lambda}\Big[(j-1)-(i-1)\Big]
=\sum\limits_{(i,j)\in\lambda}(j-i)=\sum\limits_{u\in\lambda}c_u$
and $X_\lambda+Y_\lambda+|\lambda|=\sum\limits_{(i,j)\in\lambda}\Big[(\lambda_i-j)+(\lambda'_j-i)
+1\Big]=\sum\limits_{u\in\lambda}h_u$, where $(\lambda_i-j)+(\lambda'_j-i)+1$ is clearly the hook length of box $(i,j)$.

Further functions of $\lambda$ that can be written in terms of hook lengths or contents are
$\log p_\lambda=\sum_{u\in\lambda}\log(h_u)$ making its appearance in section~3, and $D(\lambda)=\sum_{u\in\lambda}\delta_{0,c_u}$, the number of boxes of $\lambda$ on the main diagonal, that we will meet in section~4.
\begin{figure}
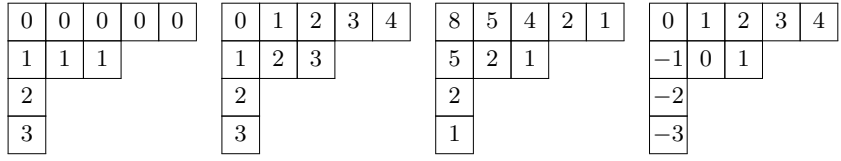

\centering\footnotesize
\ytableausetup{centertableaux}\small
\begin{ytableau}
0 & 0 & 0 & 0 & 0 \\
1 & 1 & 1 \\
2 \\
3
\end{ytableau}\quad
\begin{ytableau}
0 & 1 & 2 & 3 & 4 \\
1 & 2 & 3 \\
2 \\
3
\end{ytableau}\quad
\begin{ytableau}
8 & 5 & 4 & 2 & 1 \\
5 & 2 & 1 \\
2 \\
1
\end{ytableau}\quad
\begin{ytableau}
0 & 1 & 2 & 3 & 4 \\
-1 & 0 & 1 \\
-2 \\
-3
\end{ytableau}
\caption{The partition $(5,3,1,1)\vdash 10$, drawn as Young diagram, filled from left to right with bumping step counts, sums of box coordinates, hook lengths, and contents. The sum of the entries of the second diagram is 10 less than the sum of the hook lengths.
}
\label{fig:hooklengthsandcontents}
\end{figure}

Being able to express some function of $\lambda$ in terms of the contents or hook lengths of the boxes of $\lambda$ can allow us to employ the polynomiality results for Plancherel averages derived by Stanley~\cite{Stan10}.
\begin{thm}{\textup{(\cite[Thm.\,2.1,Thm.\,4.3]{Stan10})}}
Let $F(x)$ be a formal power series over $\mathbb{Q}$ of bounded degree that
is symmetric in the variables $x = (x_1,x_2,\ldots)$. Then both averages
$$\frac1{n!}\sum_{\lambda\vdash n}f_\lambda^2F(c_u:u\in\lambda)
\qquad\textup{and}\qquad
\frac1{n!}\sum_{\lambda\vdash n}f_\lambda^2F(h^2_u:u\in\lambda)$$
are polynomial functions of $n$.
\end{thm}
Note that an even more general result is given in \cite[Thm\,4.4]{Stan10}. See also
\cite{Ols10} for alternative proofs and further generalizations.

The proof of
\cite[Thm.\,2.1]{Stan10} restricts w.\,l.\,o.\,g.~to elementary symmetric functions indexed by partitions $\mu=(\mu_1,\ldots,\mu_k)$, i.e., to functions
$F(\cdot)=e_\mu(\cdot)=\prod_{i=1}^{k}e_{\mu_i}(\cdot)$, where
$e_m(x_1,x_2,\ldots)=\sum_{i_1<\cdots<i_m}x_{i_1}\ldots x_{i_m}$ for $m\ge1$.
As remarked in \cite{Stan10} right below that proof, the resulting polynomial $N_\mu$ is of degree $|\mu|$ if and only if $|\mu|$ is even and $\mu_1\le\frac{|\mu|}2$, otherwise, $N_\mu=0$.
Here is an immediate application of these degree considerations.
\begin{lem}$\Var (X_n-Y_n)=\binom{n}{2}$.\label{lem:bigOm}\end{lem}
\begin{proof}
Since $\EE(X_n-Y_n)=0$, we have
$$\Var (X_n-Y_n)=\frac1{n!}\sum_{\lambda\vdash n}f_\lambda^2\bigg(\sum_{u\in\lambda}c_u\bigg)^2=\frac{n(n-1)}2.$$
Note that here we have $F(c_u:u\in\lambda)=\big(e_1(c_u:u\in\lambda)\big)^2$, i.e.,
$\mu=(1,1)$. The polynomial $N_\mu$ is therefore of degree $2$, and it is completely determined by its values at $n\in\{0,1,2\}$, which are $N_\mu(0)=N_\mu(1)=0, N_\mu(2)=1$, proving the claim. The result
$N_\mu(n)=\frac{n(n-1)}2$ is also stated as a special case in \cite[p.\,94]{Stan10}.
\end{proof}
A workaround is needed for $\Var (X_n+Y_n)$, or even $\EE(X_n+Y_n)$,
because $X_n+Y_n$ is not a symmetric function of $\{h^2_u:u\in\lambda\}$, but only of $\{h_u:u\in\lambda\}$.
Finding a series representation $x=\sum_{k\ge0}a_kp_k(x^2)$ with  polynomials $p_k$, that holds for integers $x\ge1$, (but need not hold or even converge elsewhere) would allow, interchanging summations, to apply the polynomiality results termwise. If we are lucky --- and we are --- the polynomials $p_k$ have well known Plancherel averages.
Such kind of workaround is employed in this paper to deal with Plancherel averages
of several interesting functions of partitions, leading firstly to a representation of the
expectation as a binomial convolution, that is free of references to partitions, and can be analyzed using Rice's integral, or Poisson generating functions. In some cases holonomicity of the sequence of expectations can be inferred from the binomial convolution representation. This then allows for fast computation of many terms, that can be used to numerically confirm error terms, or conduct experiments.

The paper is organized as follows: In section~2 we consider the expected number of bumping steps in the RSK algorithm. In particular, we derive asymptotics for $\EE(X_n+Y_n)$, thus refining the result obtained by Romik~\cite{Rom05}. In section~3 we consider
$\EE\log \PP\! l^{(n)}(\lambda)$, with $\lambda$ distributed according to
Plancherel measure. From the first asymptotic terms we obtain a very explicit representation of a constant appearing in an \emph{almost equipartition property}
{\color{black}(abbreviated AEP)} for Plancherel measure, conjectured in~\cite{VeKe85}, and proven in~\cite{Buf12}. In section~4 we derive asymptotics of the expectation of the side length $D(\lambda)$ of the \emph{Durfee square} of $\lambda$, i.e., the largest square fitting in the upper left corner of the Young diagram of $\lambda$, when partitions $\lambda$ are distributed according to
Plancherel measure, see Table~\ref{tab:la4}.
Considering in section~5 more generally lengths of south-east directed cuts through the Young diagram of $\lambda$,
we enter the realm of a sequence of random curves $\psi_\lambda$ known to converge uniformly in probability to the \emph{Logan-Shepp-Vershik-Kerov limit shape curve} $\Omega$ for $|\lambda|\to\infty$.
For any fixed integer $a$ the sequence with terms $\sqrt{n}\EE\psi_\lambda\left(\frac a{\sqrt{2n}}\right)$, with expectation computed with respect to $\PP\! l^{(n)}$, turns out to be holonomic.
Experiments then strongly hint at convergence of $\EE\psi_\lambda\left(\frac {\lfloor u\sqrt{2n}\rfloor}{\sqrt{2n}}\right)\to\Omega(u)$, uniformly in $u$, and reveal that second order terms show interesting fluctuations. However we are only able to prove asymptotic results in the case of fixed $a$, i.e., in the vicinity of $u=0$. In section~6 we return to the number $Y_n$ of bumping steps, giving a heuristic argument for $\Var Y_n=\Oh(n^2)$, based on the limit shape curve.

\begin{table}
\renewcommand{\arraystretch}{1.5}
\centering
\caption{Some functions of partitions $\lambda\vdash4$ and their expectations with respect to Plancherel measure. }\label{tab:la4}
\ytableausetup{centertableaux}
\begin{tabular}{|c||c|c|c|c|c||c|}
  \hline
  $\lambda$  &  \tiny\ydiagram{4}  &   \tiny\ydiagram{3,1}  &
  \tiny\ydiagram{2,2}  &  \tiny\ydiagram{2,1,1}  &
  \tiny\ydiagram{1,1,1,1}&\EE\\[.1cm]
  \hline
  $\PP\! l^{(4)}(\lambda)$& $\frac1{24}$&$\frac3{8}$& $\frac1{6}$&$\frac3{8}$& $\frac1{24}$&\\
  $X_\lambda-Y_\lambda$& $6$&$2$& $0$&$-2$& $-6$&$0$\\
  $X_\lambda+Y_\lambda$& $6$&$4$& $4$& $4$& $6$&$\frac{25}6$\\
  $\log\PP\! l^{(4)}(\lambda)$& $\log\frac1{24}$&$\log\frac3{8}$& $\log\frac1{6}$&$\log\frac3{8}$& $\log\frac1{24}$&$\tfrac{8}3\log\frac12\!+\!\tfrac{\log3}2$\\
  $D(\lambda)$& $1$& $1$& $2$&$1$& $1$&$\frac76$\\
  \hline
\end{tabular}
\end{table}

\section{Refined asymptotics of the expected number of bumping steps in the RSK algorithm
}
Recall that $\EE (X_n+Y_n)=2\EE Y_n$ denotes twice the expected number of bumping steps of the RSK algorithm when applied to a random permutation of $\{1,\ldots,n\}$.
Romik \cite[eq.\,(1)]{Rom05} derived the following asymptotic result, $\EE Y_n\sim \frac{\color{black}128}{27\pi^2}n^{\frac32}$, and showed $Y_n/\EE Y_n\to1$ in probability.
The sequence of interest starts $$\big(\EE (X_n+Y_n)\big)_{n=1}^{10}=\big(0,1, \tfrac73, \tfrac{25}{6}, \tfrac{19}{3}, \tfrac{44}{5}, \tfrac{347}{30}, \tfrac{8181}{560}, \tfrac{541273}{30240}, \tfrac{1943453}{90720}\big).$$
The next theorem leads to a refinement of Romik's asymptotic equivalent for $\EE Y_n$.
\begin{thm}\label{t:Romik}{\textup{(Expected number of bumping steps in the RSK algorithm)}} Let $\delta_n:=\log n+2\gamma+12\log 2$, with $\gamma$ denoting Euler's constant. Then
\begin{align*}\EE (X_n+Y_n)&=\frac{256}{27\pi^2}n^{\frac32}-n
+\frac{9\delta_n-77}{9\pi^2}n^{\frac12}
+\frac{3510\delta_n-31589}{27648\pi^2}n^{-\frac12}\\
&\hphantom{==}+\frac{5565\delta_n-62224}{786432\pi^2}n^{-\frac32}
+\frac{e^8}{2^{12}\pi^{3/2}}\cos\left(8\sqrt{n}+\frac{\pi}4\right)n^{-\frac74}+\Oh\Big(n^{-\frac94}\Big).\end{align*}
\end{thm}
\begin{proof}
As $X_n+Y_n$ is not a symmetric polynomial of the multiset $\{h_u^2:u\in\lambda\}$, but only of the multiset $\{h_u:u\in\lambda\}$, we can not expect $\EE (X_n+Y_n)$ to be a polynomial in $|\lambda|$. Indeed, by Romik's result, $\EE (X_n+Y_n)=\Theta(n^\frac32)$ is definitely not a polynomial.
However, we can invoke polynomiality results via the following identity. Using
\begin{equation}
p(x,r):=\prod_{i=1}^r(x^2-i^2),\label{e:p}\end{equation}
the equation
\begin{equation}x=1+\sum_{r=1}^\infty \binom{2r}{r}\frac{(-1)^{r}}{(1-2r)(2r+1)!}p(x,r)\label{e:id}\end{equation}
holds for $x\in\mathbb{N}:=\{1,2,3,\ldots\}$. This will be proved in the appendix.

Now, by \cite[Thm.\,1]{Pan12}, we have
\begin{equation}\frac1{n!}\sum_{\lambda\vdash n}f_\lambda^2\sum_{u\in\lambda}p(h_u,r)=K_r\binom{n}{r+1},\label{e:pan}\end{equation}
with $K_r=\frac{(2r)!(2r+1)!}{(r+1)!^2r!}$,
leading to
\begin{align*}\EE\,(X_n&+Y_n)=\frac1{n!}\sum_{\lambda\vdash n}f_\lambda^2\sum_{u\in\lambda}(h_u\!-\!1)
\overset{\eqref{e:id}}{=}\frac1{n!}\sum_{\lambda\vdash n}f_\lambda^2\sum_{u\in\lambda}\sum_{r=1}^\infty \binom{2r}{r}\frac{(-1)^{r}}{(1-2r)(2r+1)!}p(h_u,r)\\
&\overset{\eqref{e:pan}}{=}\sum_{r=1}^\infty \binom{2r}{r}\frac{(-1)^{r}}{(1-2r)(2r+1)!}K_r\binom{n}{r+1}\\
&=\sum_{r=1}^\infty\binom{2r}{r}^2\frac{(-1)^{r+1}}{(2r-1)(r+1)(r+1)!}\binom{n}{r+1}
=\sum_{r=2}^n\binom{2r-2}{r-1}^2\frac{(-1)^{r}}{(2r^2-3r)r!}\binom{n}{r}\\
&=(-1)^n\frac{n!}{2\pi \i}\oint_C\frac{f(z)}{z(z-1)(z-2)\cdots(z-n)}dz,
\end{align*}
where $C$ is a contour that encircles integers $2,3,\ldots,n$, but neither
any other integers, nor poles of $f$, which is given by
$$f(z):=\frac{\Gamma^2(2z-1)}{z^2(2z-3)\Gamma^5(z)}.$$
For the method of evaluating a large finite difference via the so-called \emph{Rice's integral} used above see the article \cite{FlS95}.
By computing (leading asymptotic terms of) residues of $g_n(z):=f(z)\frac{n!\Gamma(-z)}{\Gamma(n+1-z)}$
at $\pm\frac32,1,\pm\frac12$, (note that $g_n(z)$ is analytic for
$z\in\{-1,0\}$)
we obtain
\begin{align*}
\EE\,(X_n+Y_n)=&\,\frac{256 n^{\frac32}}{27\pi^2}\left(1-\frac{3}{8n}
-\frac{7}{128n^2}-\frac{9}{1024n^3}\right)\\
&-n\\
&+\delta_n\frac{n^{\frac12}}{\pi^2}\Big(1+\frac{1}{8n}+\frac{1}{128n^2}\Big)-
\frac{5 n^{\frac12}}{\pi^2}\left(1+\frac{1}{8n}
-\frac{1}{1920n^2}\right)\\
&+\delta_n\frac{n^{-\frac12}}{512\pi^2}\Big(1-\frac{3}{8n}\Big)
+\frac{n^{-\frac12}}{1024\pi^2}\left(1+\frac{13}{8n}\right)\\
&-\delta_n\frac{n^{-\frac32}}{2^{18}\pi^2}
+\frac{n^{-\frac32}}{2^{14}\!\cdot\!3\pi^2}\\
&+\Oh\left(n^{-\frac52}\log n\right)+\frac{1}{2\pi \i}\oint_{C'}f(z)\frac{n!\Gamma(-z)}{\Gamma(n+1-z)}dz
,
\end{align*}
where we recall $\delta_n=\log n+2\gamma+12\log 2,$ and
where $C'$ encircles the interval $[-\frac32,n]$, but no poles of the integrand outside that interval. This integral is taken care for in the appendix, yielding the contribution involving $\cos(8\sqrt{n}+\frac\pi4)$. Such a term is not completely uncommon, see e.g. the example at the end of \cite[sec.\,5]{FlS95}. Slightly rearranging the terms completes the proof.
\end{proof}

\subsection{Holonomicity of the sequence $(\EE(X_n+Y_n)+n)_{n\in\mathbb{N}}$}
 Defining
$$u_n:=\sum_{r=1}^n\binom{2r-2}{r-1}^2\frac{(-1)^{r}}{(2r^2-3r)r!}\binom{n}{r},$$
we have $\EE (X_n+Y_n)=u_n-n$. The sequence $(u_n)$ satisfies the linear recurrence relation
$$u_{n+4}
=\tfrac{4n^3+23n^2+63n+78}{(n+3)(n+4)^2}u_{n+3}
-\tfrac{(2n-1)(3n+2)}{(n+4)^2}u_{n+2}
+\tfrac{(4n-7)(n+2)}{(n+4)^2}u_{n+1}
-\tfrac{(n+1)(n+2)}{(n+4)^2}u_n
,$$
with initial conditions
$$u_0=0,\quad u_1=1,\quad u_2=3,\quad u_3=\frac{16}3.$$
Clearly, the terms $\frac{u_n}{n!}$ comprise a sequence, that is the convolution of two
sequences that are obviously holonomic. For said convolution the \emph{gfun} package \cite{SZ94} then easily produces a recursion.

For the Poisson generating function
$U(z):=e^{-z}\sum_{n\ge0}\frac{u_n}{n!}z^n$ we obtain
$$U(z)=-z\sum_{k\ge0}\binom{2k}{k}^2\frac{(-z)^k}{(k+1)(2k-1)(k+1)!^2}=z\,\setlength\arraycolsep{1pt}
{}_2 F_3\left[\begin{matrix}-\frac12,\,
\frac12~\\2,\,2,\,2\end{matrix};-16z\right],$$
a hypergeometric function that may also be used to recover
asymptotics of $u_n$, see
\cite[Sec.\,5.11.2]{Luk69} for asymptotic expansions of generalized hypergeometric functions.
Indeed, the leading terms of an asymptotic expansion of $U(z)$, provided by Maple, together with Depoissonization via the saddle point method, yield
an alternative proof of Theorem~\ref{t:Romik}.

Note that the recurrence relation allows for easily computing millions of terms of the sequence $(\EE(X_n+Y_n))_{n\ge1}$, which can be used to numerically confirm the error term in Theorem~\ref{t:Romik}.

\section{The constant appearing in the AEP for Plancherel measure}
We consider the random variables $Z_n=Z_n(\lambda):=\sum_{u\in\lambda}\log h_u$, where $\lambda\vdash n$ is distributed according to the Plancherel measure, and denote $z_n:=\EE Z_n$.
The sequence starts
\begin{align*}(z_1,\ldots,z_{5})&=(0,~\log2,~\tfrac{\log2}3\!+\!\log3,~\tfrac{17}6\log2\!+\!\tfrac{\log3}4,~\tfrac{13}6\log2\!+\!\tfrac7{10}\log3\!+\!\tfrac7{12}\log5)\\
&\approx(0,~ 0.6931471806,~ 1.329661349,~ 2.238570083,~ 3.209686276)
\end{align*}
The first few asymptotic terms of
$z_n$ will lead to a representation of the constant $H$, conjectured by Vershik and Kerov \cite{VeKe85} to exist as the limit in probability of random variables $-\frac1{\sqrt{n}}\log \PP\! l^{(n)}(\lambda)$, where $\PP\! l^{(n)}(\lambda):=n!\left(\prod_{u\in\lambda}\frac1{h_u}\right)^2$, with $\lambda\vdash n$ again distributed according to the Plancherel measure.
A strengthening of the conjecture (convergence in $L_p$ for $p<\infty$) has been proved by Bufetov \cite{Buf12}, from which we borrowed above notation,
and an expression for $H$ in terms of a threefold integral has been given
in \cite[eq.\,(15)]{Buf12}. We aim here at a less involved representation of $H$, and at more terms of an asymptotic expansion
of $\EE[n^{-\frac12}(2Z_n-\log n!)]$.
\begin{thm}\label{t:Buf} Let {\color{black} $H_n=1+\frac12+\cdots+\frac1n$} denote the $n$th harmonic number. Then, as $n\to\infty$, we have
$$
-\frac{\EE\log \PP\! l^{(n)}(\lambda)}{\sqrt{n}}=H-\left(\frac{13}{24}\log n+\frac{13\gamma}{12}+\log\sqrt{2\pi}+\frac14-h'(0)\right)\frac1{\sqrt{n}}+o\big(n^{-\frac12}\big),
$$
where
\begin{align*}H&=\frac{16}{3\pi^2}(4\gamma+1)+\frac{64}{\pi^2}
\sum_{\ell\ge2}\frac{\ell^2}{4\ell^2-1}\Big(\log \ell-H_\ell+\gamma+\frac1{2\ell}\Big)
\\
&\approx1.87702830628\end{align*}
and
$$h'(0)=\sum_{\ell\ge2}\ell\left(H_\ell-\log \ell-\gamma- \frac1{2\ell}+\frac1{12\ell^2}\right)\approx0.001562493.$$
\end{thm}
\begin{proof}
As we will prove in the appendix, the Kronecker delta defined on $\mathbb{N}\times\mathbb{N}$  can be expressed in terms of the polynomials $p(x,r)$ given in \eqref{e:p} as follows,
\begin{equation}\delta_{\ell,n}=\sum_{r=\ell-1}^\infty (-1)^{\ell+r+1}\frac{2\ell^2}{(r+\ell+1)!(r-\ell+1)!}p(n,r).\label{e:kron}\end{equation}
From this we deduce
$$\log n=\sum_{\ell\ge2}\log \ell\sum_{r\ge\ell-1}\frac{2(-1)^{\ell+r+1}\ell^2}{(r+\ell+1)!(r-\ell+1)!}p(n,r)=2\sum_{r\ge2}(-1)^{r}g(r)p(n,r-1),$$
for $n\in\mathbb{N}$, where $g$ is given by
$$g(r)=\sum_{\ell=2}^r\frac{(-1)^{\ell}\ell^2\log \ell}{\Gamma(r+\ell+1)\Gamma(r-\ell+1)}.$$
We want to extend $g$ to a meromorphic function in the right halfplane $\Re r>-1$.
Therefore we employ
$$\log \ell=H_\ell-\gamma-\frac1{2\ell}+\frac1{12\ell^2}+\Oh(\ell^{-4}),$$
and the identities (all with easy proofs, only the last one is proven in the appendix)
\begin{subequations}\begin{align}\sum_{\ell=2}^r\frac{(-1)^{\ell}\ell^2}{\Gamma(r+\ell+1)\Gamma(r-\ell+1)}&=\frac1{\Gamma(r)\Gamma(r+2)}\\
\sum_{\ell=2}^r\frac{(-1)^{\ell}\ell}{\Gamma(r+\ell+1)\Gamma(r-\ell+1)}&=\frac{3(r-1)}{2(2r-1)\Gamma(r)\Gamma(r+2)}\\
\sum_{\ell=2}^r\frac{(-1)^{\ell}}{\Gamma(r+\ell+1)\Gamma(r-\ell+1)}&=\frac {r-1}{2r\Gamma(r)\Gamma(r+2)}\\
\sum_{\ell=2}^r\frac{(-1)^{\ell}\ell^2(H_\ell-1)}{\Gamma(r+\ell+1)\Gamma(r-\ell+1)}&
=\frac{1}{4(r-1)(2r-1)\Gamma(r)^2}.\label{e:numberd}\end{align}\end{subequations}
By Euler's reflection formula, for complex $r\not\in\mathbb{Z}$ and for real $\ell\to\infty$ we have
$$\Gamma(r+\ell+1)\Gamma(r-\ell+1)\frac{\sin\pi(\ell-r)}{\pi}=
\frac{\Gamma(r+\ell+1)}{\Gamma(\ell-r)}
\sim\ell^{2r+1}.
$$
Therefore the following series
$$h(r):=\sum_{\ell\ge2}(-1)^{\ell}\ell^2\frac{\log \ell-H_\ell+\gamma+\frac1{2\ell}-\frac1{12\ell^2}}{\Gamma(r+\ell+1)\Gamma(r-\ell+1)}$$
converges for $\Re r>-1$, and satisfies $h(1)=h(0)=0$, with
$h'(0)$ as given in the theorem.
Hence
$$g(z)=h(z)+\frac{1-\gamma}{\Gamma(z)\Gamma(z+2)}+\frac{1}{4(z-1)(2z-1)\Gamma(z)^2}
-\frac{(z-1)(16z+1)}{24z(2z-1)\Gamma(z)\Gamma(z+2)}
$$
is the sought extension, meromorphic for $\Re z>-1$. Now
\begin{align*}z_n=\frac1{n!}\sum_{\lambda\vdash n}f_\lambda^2\sum_{u\in\lambda}\log h_u&=2\sum_{r\ge2}(-1)^{r}g(r)\frac1{n!}\sum_{\lambda\vdash n}f_\lambda^2\sum_{u\in\lambda}p(h_u,r-1)\\
&=2\sum_{r\ge2}(-1)^{r}g(r)K_{r-1}\binom{n}{r}\\
&=(-1)^n\frac{n!}{2\pi \i}\oint_C\frac{\phi(z)}{z(z-1)(z-2)\cdots(z-n)}dz,
\end{align*}
where
$$
\phi(z)=2g(z)\frac{\Gamma(2z)\Gamma(2z-1)}{\Gamma(z+1)^2\Gamma(z)}.
$$
Here $C$ is a contour that encircles integers $2,\ldots,n$, but neither
any other integers, nor poles of $\phi$.

By computing (leading asymptotic terms of) residues of $\Phi_n(z):=\phi(z)\frac{n!\Gamma(-z)}{\Gamma(n+1-z)}$
at $1,\,\frac12$, and $0$,
we obtain
\begin{align*}
z_n=&\,\frac{n\log n-n}2-\frac14
+\frac{4}{9\pi^2}\Big(24\gamma+7-18\pi h(\tfrac12)\Big)n^{\frac12}\\
&-\frac{\log n}{48}-\frac{13\gamma}{24}+\frac18+\frac{h'(0)}2
+\Oh(n^{-\frac12})+\frac{1}{2\pi \i}\oint_{C'}\Phi_n(z)dz,
\end{align*}
where $C'$ encircles the interval $[0,n]$, but no poles of the integrand outside that interval. As shown in the appendix, the latter integral is $o(1)$, thus we arrive at
\begin{align*}-\EE\frac{\log \PP\! l^{(n)}(\lambda)}{\sqrt{n}}&=\EE\frac{2Z_n-\log n!}{\sqrt{n}}\\
&=\frac{8}{9\pi^2}\Big(24\gamma+7-18\pi h(\tfrac12)\Big)
-\frac{13}{24}\frac{\log n}{\sqrt{n}}\\
&\ \ \ \ -\left(\frac{13\gamma}{12}+\log\sqrt{2\pi}+\frac14-h'(0)\right)\frac1{\sqrt{n}}
+o\big(n^{-\frac12}\big).
\end{align*}
Finally, for the evaluation of $h(\frac12)$, we use $\Gamma(\frac12+\ell+1)\Gamma(\frac12-\ell+1)=(-1)^{\ell+1}\frac\pi4(4\ell^2-1)$,
as well as $\sum_{\ell\ge2}(4\ell^2-1)^{-1}=\frac16$, leading to
$$h(\tfrac12)=-\frac{4}{\pi}\sum_{\ell\ge2}\frac{\ell^2}{4\ell^2-1}\Big(\log \ell-H_\ell+\gamma+\frac1{2\ell}\Big)+\frac1{18\pi},$$
which completes the proof.
\end{proof}
\begin{table}
\renewcommand{\arraystretch}{1.35}
\centering
\caption{Some terms of the sequence approaching $H$. }\label{tab:H}
\begin{tabular}{|c|c|c|c|c|c|}
  \hline
 $n$ & $\frac{1}{\sqrt{n}}(2z_n-\log n!)$& $n$ & $\frac{1}{\sqrt{n}}(2z_n-\log n!)$ &$n$ & $\frac{1}{\sqrt{n}}(2z_n-\log n!)$\\
  \hline
  $2$& $0.4901290717$&$7$& $0.8208116414$&$128$& $1.4880650932$\\
  $3$& $0.5008878635$&$8$& $0.8690239552$&$256$& $1.5781760349$\\
  $4$& $0.649543169$&$16$& $1.0657023619$& $512$& $1.6489336120$\\
  $5$& $0.7297992837$&$32$& $1.2347905493$&$1024$& $1.7039138626$\\
  $6$& $0.7726513179$& $64$& $1.3748129422$&$2048$& $1.7462734777$\\
  \hline
\end{tabular}
\end{table}
\begin{rem}
Note that the term $2\sum_{r\ge2}(-1)^{r}g(r)K_{r-1}\binom{n}{r}$ can be used
to compute $z_n$ for values of $n$ so large that naively generating all partitions
$\lambda\vdash n$ is not an option. Of course, care has to be taken, since cancellations will occur in numerical computations because of alternating signs of summands. Table~\ref{tab:H} shows that
$\frac{1}{\sqrt{n}}(2z_n-\log n!)$ is slowly approaching $H$ from below, with
the values obtained in \cite[Table\,1]{VePa10} from simulations fitting neatly into this pattern.
The convergence rate is in good accordance with the error term given in Theorem~\ref{t:Buf}.
Observe that $\frac{13\gamma}{12}+\log\sqrt{2\pi}+\frac14-h'(0)\approx1.792693.$ For $n=2048$ we then get $H-(\frac{13}{24}\log{2048}+1.792693)/\sqrt{2048}\approx1.746154$,
which matches the table entry fairly well.
\end{rem}

\section{The expected side length of the Durfee square}
Here we consider the side length of the Durfee square of partition $\lambda$,
$$D(\lambda):=\max\{i:\lambda_i\ge i\},$$ and we denote the restriction of
$D(\lambda)$ to $\lambda\vdash n$ distributed according to Plancherel measure by $D_n$.
With respect to uniform measure, where all partitions $\lambda\vdash n$ are equally likely, the expectation and the most likely value of
$D(\lambda)$ have been studied in \cite{CCS98,Can05,Mut02}. Regarding Plancherel measure, it is known since the days of the limit shape theorem (see Theorem~\ref{VKLS} in the next section) that $\frac1{\sqrt{n}}D_n\to\frac2\pi$ in probability. Furthermore, we may deduce from
\cite[Thm.\,3.6]{Bog07}
convergence in distribution of  $\frac{\pi}{\sqrt{\log n}}\left(D_n-\frac2\pi \sqrt{n}\right)$ to a standard normal random variable. Should that convergence in distribution be accompanied by convergence of second moments,
$\EE D_n=\frac2\pi \sqrt{n}+\Oh\left(\sqrt{\log n}\right)$ would follow.
We are not aware of a proof of such result, let alone of any results in the literature regarding fine asymptotics of $\EE D_n$.

\begin{thm}\label{t:Durf} Let $d_n:=\EE D_n$. Then, as $n\to\infty$, we have
$$d_n=\frac2\pi\sqrt{n}+\left(\frac{3}{16\pi}
-\frac{e^2}{8\pi}\sin\left(4\sqrt{n}\right)\right)\frac1{\sqrt{n}}+\Oh\big(n^{-1}\big).$$
\end{thm}
\begin{proof}
In terms of contents $c_u$ of a Young diagram $\lambda$, we have $$D(\lambda)=\sum_{u\in\lambda}\delta_{0,c_u}.$$
Define polynomials in terms of the polynomials $p(x,r)$ given in \eqref{e:p} via
$$q(x,r):=\prod_{i=0}^{r-1}(x^2-i^2)=\begin{cases}x^2p(x,r-1),&r\ge1,\\ 1,&r=0.\end{cases}$$
These also allow for a representation of the Kronecker delta, similar to \eqref{e:kron},
\begin{equation}\delta_{\ell,n}=\sum_{r=\ell}^\infty (-1)^{\ell+r}\frac{2-\delta_{0,n}}{(r+\ell)!(r-\ell)!}q(n,r),\label{e:kron2}\end{equation}
now valid for non-negative integers $\ell,n$.
By \cite[eq.\,(7)]{Stan10}, see also \cite[Thm.\,A.1]{FKM08}, we have
\begin{equation}\frac1{n!}\sum_{\lambda\vdash n}f_\lambda^2\sum_{u\in\lambda}q(c_u,r)=\frac{(2r)!}{(r+1)!}\binom{n}{r+1}.
\label{e:fuj}\end{equation}
This leads to the representation
\begin{equation}d_n=\frac1{n!}\sum_{\lambda\vdash n}f_\lambda^2\sum_{u\in\lambda}\delta_{0,c_u}=
\sum_{r\ge1}\frac{(-1)^{r+1}(2r-2)!}{(r-1)!^2r!}\binom{n}{r}.\label{e:durfsum}\end{equation}
For the Poisson generating function
$D(z):=e^{-z}\sum_{n\ge0}\frac{d_n}{n!}z^n$ we obtain
$$D(z)=z\sum_{k\ge0}\frac{(-z)^k(2k)!}{k!^2(k+1)!^2}
=z\left(2J_0^2(2\sqrt{z})-\frac{J_0(2\sqrt{z})J_1(2\sqrt{z})}{\sqrt{z}}+2J_1^2(2\sqrt{z})\right),$$
where $J_0$ and $J_1$ are Bessel functions of the first kind. We may use
$D(z)$ to recover
asymptotics of $d_n$, see
\cite[Sec.\,5.11.4]{Luk69} for asymptotic expansions of Bessel functions.
We find
$$D(z)=\frac2\pi\sqrt{z}-\frac1{16\pi\sqrt{z}}-\frac{\sin(4\sqrt{z})}{8\pi\sqrt{z}}+\frac{3\cos(4\sqrt{z})}{64\pi z}+\Oh\big(z^{-\frac32}\big),$$
for $|z|\to\infty$, $|\arg z|\le\pi-\delta$ with $\delta>0$. A uniform bound is furnished by $|D(z)|\le\cosh(4\sqrt{|z|})$.
Evaluating now $d_n=\frac{n!}{2\pi\i}\oint_C z^{-n-1}e^zD(z)dz$, with contour $C:=\{z\in\mathbb{C}:|z|=n\}$,
observing that there is an approximate saddle point at $z=n$, finishes the proof.
\end{proof}
\begin{rem}
The sequence starts $\big(d_n\big)_{n=1}^{10}\!=\!(1, 1, 1, \frac76, \frac{17}{12}, \frac{33}{20}, \frac{109}{60}, \frac{3217}{1680}, \frac{39703}{20160}, \frac{364859}{181440})$,
and it again satisfies a linear recurrence relation,
\begin{equation}\label{e:holdurf}d_{n+3}
=\frac{3n^2+9n+8}{(n+2)(n+3)}d_{n+2}
-\frac{3n+1}{n+3}d_{n+1}
+\frac{n+1}{n+3}d_{n},\end{equation}
with initial conditions
$$d_0=0,\quad d_1=1,\quad d_2=1,$$
readily obtained from \eqref{e:durfsum} using \emph{gfun}.
\end{rem}

\begin{figure}
\centering
\includegraphics[scale=1.0]{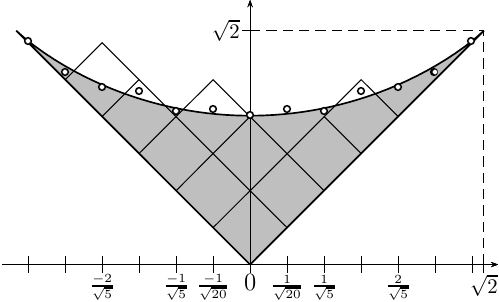}
\caption{The Logan--Shepp--Vershik--Kerov limit shape curve $\Omega$ (upper boundary of the grey region). Superimposed is the partition $(5,3,1,1)\vdash10$, properly scaled and rotated by $135^{\circ}$.
The dots have coordinates $\Big(\frac a{\sqrt{20}},{\frac1{\sqrt 5}}\big(\omega_{a,10}+\frac {|a|}2\big)\Big)$, for $a\in\{-6,\ldots,6\}$.
}
\label{fig:limshapcrop}
\end{figure}

\section{Expected fluctuations around the limit shape curve}

Let us introduce the \emph{limit shape curve}
\begin{equation}
\Omega(u)=\begin{cases}\frac2\pi\left(u\arcsin\frac u{\sqrt{2}}+\sqrt{2-u^2}\right),&\textup{ if }|u|\le\sqrt{2},\\ |u|, &\textup{ if }|u|>\sqrt{2}.
\end{cases}\label{e:lish}\end{equation}

The lower right boundary of the Young diagram of a partition $\lambda\vdash n$, scaled to have unit area, rotated together with parts of positive $x$-axis and negative $y$-axis by $135^\circ$, gives rise to a piecewise linear function $\psi_{\lambda}$, also defined on $\mathbb{R}$.
When $\lambda$ is distributed according to Plancherel measure, the random
functions $\psi_{\lambda}$ approach the limit shape curve, as $n\to\infty$,
in a sense that is made precise in the following result by Vershik and Kerov~\cite{VeKe77} and Logan and Shepp~\cite{LoS77}, which we present following closely~\cite[Thm.\,1.22]{Rom15}.
\begin{thm} \label{VKLS}{\textup{(Limit shape theorem for Plancherel-random partitions)}}
For all $\eps>0$, we have $\mathbb{P}(\|\psi_\lambda-\Omega\|_\infty>\eps)\to0$ as $n\to\infty$, i.e.,\,the random functions $\psi_\lambda$ converge to $\Omega$ in probability in the norm $\|\cdot\|_\infty$.
\end{thm}

A discretized version of $\psi_{\lambda}$, defined on the set $\big\{\frac a{\sqrt{2n}}:a\in\mathbb{Z}\big\}$, can be expressed in terms of contents $c_u$ of $\lambda$ via $\Psi_\lambda(a):=\sum_{u\in\lambda}
\delta_{{\color{black}-}a,c_u}$.

Indeed, the set $\Big\{\Big(\frac a{\sqrt{2n}},\frac 2{\sqrt{2n}}\big(\Psi_\lambda(a)+\frac {|a|}2\big)\Big):a\in\mathbb{Z}\Big\}$
is a subset of the graph of $\psi_\lambda$ containing, among others, all the points where the slope of $\psi_\lambda$ changes from $1$ to $-1$ or back.
For example, if $\lambda=(5,3,1,1)$,
then $(\Psi_\lambda(a))_{a=-4}^4=(1,1,1,2,2,1,1,1,0)$. Define now a related function, $\Phi_\lambda(a):=\frac12\big(\Psi_\lambda(a)+\Psi_\lambda(-a)\big)$, i.e.,
$$\Phi_\lambda(a)=\begin{cases}D(\lambda),&a=0,\\
\frac12\sum_{u\in\lambda}
\delta_{|a|,|c_u|},&\textup{else.}\end{cases}$$
This symmetrised function is used because it can be expressed in terms of Kronecker deltas restricted to pairs of nonnegative integers, thus allowing to use the representation \eqref{e:kron2}.
Next, let
 $\omega_{a,n}:=\EE\Phi_\lambda(a)$, with $\lambda\vdash n$ distributed according to the Plancherel measure, and define a sequence of functions
$$\tilde\Omega_n(u):=\sqrt{\frac2n}\left(\omega_{\lfloor\sqrt{2n}u\rfloor,n}
+\frac12|\lfloor\sqrt{2n}u\rfloor|\right),$$
that one would expect to converge to $\Omega(u)$, although such convergence is not implied by Theorem~\ref{VKLS}.
By \cite[Thm.\,3.6]{Bog07} we have convergence in distribution
of $\frac{\pi\sqrt{n}}{\sqrt{2\log n}}\left[\sqrt{\frac2n}\Big(\Phi_\lambda(\lfloor\sqrt{2n}u\rfloor)
+\frac12|\lfloor\sqrt{2n}u\rfloor|\Big)-\Omega(u)\right]$
to a standard normal random variable in case that $|u|<\sqrt{2}$.
Should there also be convergence of second moments,
$\tilde\Omega_n(u)=\Omega(u)+\Oh\left(\sqrt{\frac{\log n}n}\right)$ would follow for $|u|<\sqrt{2}$.
See Figure~\ref{fig:limshapcrop} for the limit shape curve, and, scaled to unit area, a superimposed partition of 10, and
values $\tilde\Omega_{10}(u)$ for $u\in\{\frac a{\sqrt{20}}:-6\le a\le6\}$.
There is a seeming coincidence on the $y$-axes, yet $\Omega(0)=\frac{2\sqrt{2}}{\pi}\approx.900316$,
$\frac{\omega_{0,10}}{\sqrt{5}}=\frac{364859}{181440\sqrt{5}}\approx.899305$, and the ordinate of the upper corner of the rotated Young diagram, $\frac{2}{\sqrt{5}}\approx.894427$, are all different.
\begin{figure}
\centering
\begin{tabular}{cc}
\!\!\!\!\!\!\includegraphics[scale=0.32]{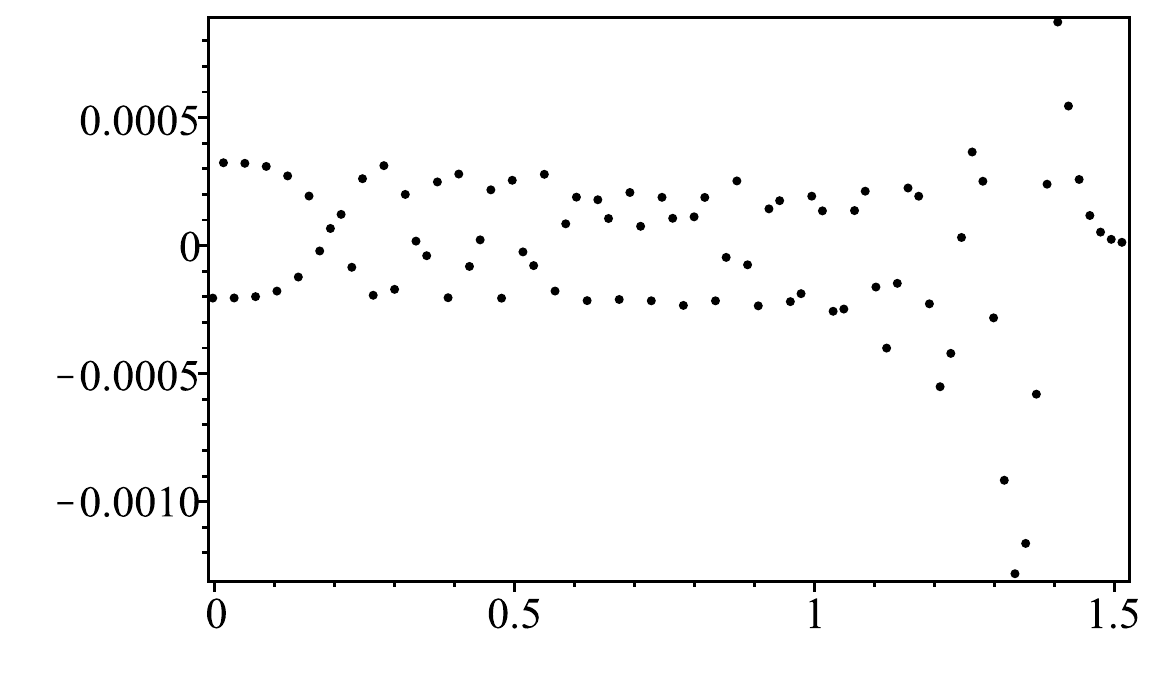}&
\includegraphics[scale=0.32]{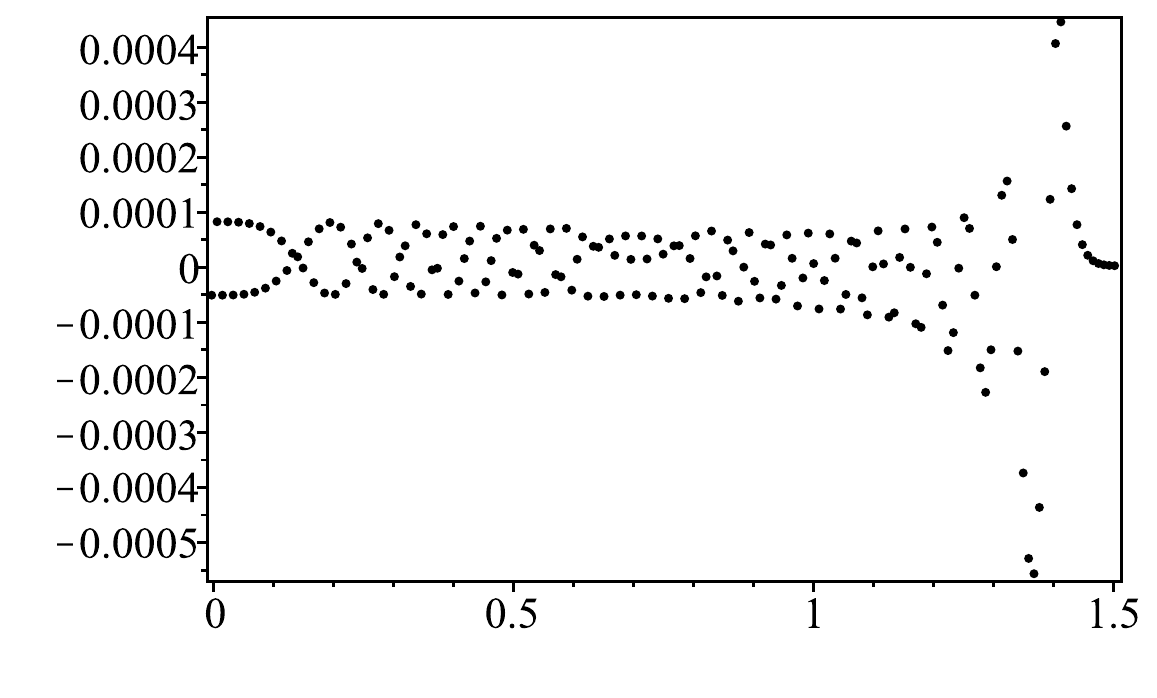}\\
\!\!\!\!\!\!\includegraphics[scale=0.32]{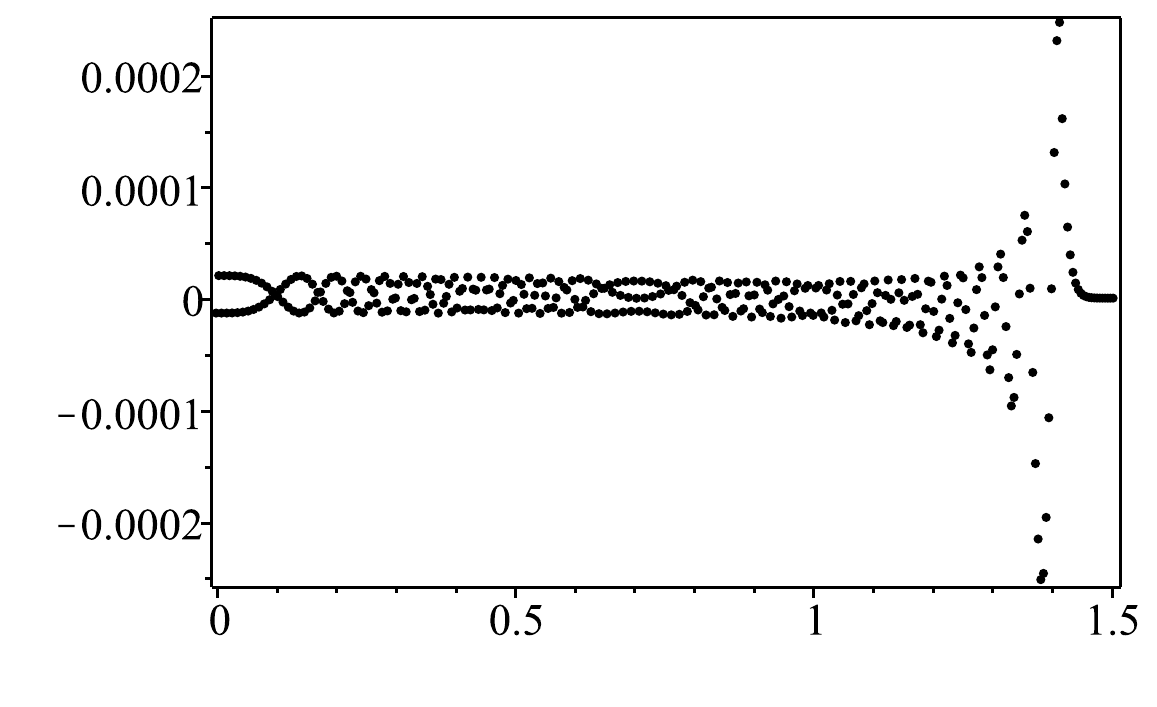}&\!\!\!\!
\includegraphics[scale=0.33]{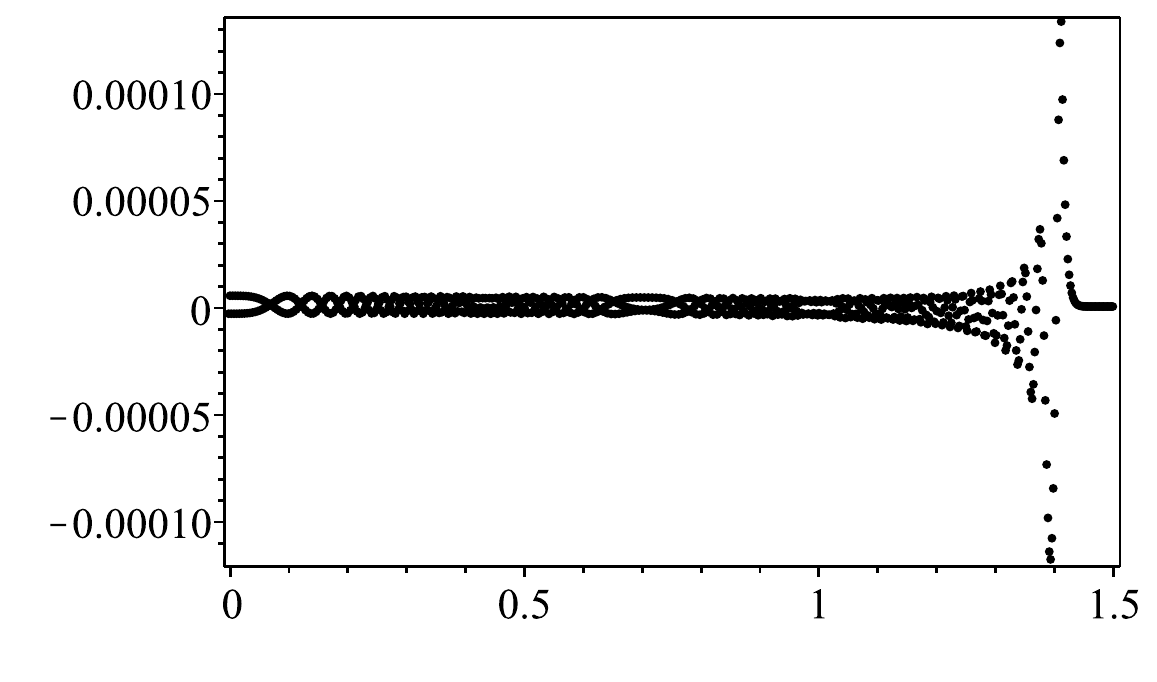}
\end{tabular}
\caption{Plots of $\tilde\Omega_n(u)-\Omega(u)$ for roughly quadrupling values of $n\in\{1573,6230,24798,98943\}$ from top left to bottom right.
}
\label{fig:OmConv}
\end{figure}

An alternating sum representation of $\omega_{a,n}$, building upon \eqref{e:kron2}, is the following
\begin{equation}\label{oman}\omega_{a,n}=\sum_{r>a}\frac{(-1)^{r+a+1}(2r-2)!}{(r-1+a)!(r-1-a)!r!}\binom{n}{r},
\end{equation}
which again gives rise to a linear recurrence relation (obtained using \emph{gfun})
\begin{align*}(n+4)(n+a+3)(n-a+3)\omega_{a,n+4}&=[4n^3+32n^2+(86-2a^2)n+78-7a^2]\omega_{a,n+3}\\
&\hphantom{==}-(n+3)(6n^2+22n+20-a^2)\omega_{a,n+2}\\
&\hphantom{==}+(n+1)(n+2)(n+3)(4\omega_{a,n+1}-\omega_{a,n}),\end{align*}
holding for $n\ge a-2$, with initial conditions
$$\omega_{a,n}=0,\textup{ for }\min(0,a-2)\le n\le a,\quad \omega_{a,a+1}=\frac1{(a+1)!}.$$
Note that there is a common factor $(n+1)$ in the recurrence relation, when $a=2$. Note also that setting $a=0$ yields a recurrence relation with both order and degree one larger than the one given in \eqref{e:holdurf}.

As was done for $d_n$, asymptotics via Poisson generating functions (which can again be expressed in terms of Bessel functions) can be obtained also for $\omega_{a,n}$ for fixed integer $a>0$:
\begin{equation}\label{omegarec}\omega_{a,n}+\frac a2=\frac2\pi\sqrt{n}+\left(\frac{4a^2+3}{16\pi}-(-1)^a\frac{e^2}{8\pi}\sin\left(4\sqrt{n}\right)\right)\frac1{\sqrt{n}}+\Oh\big(n^{-1}\big).\end{equation}
In order to obtain asymptotics of $\omega_{a,n}$ for $n$ and $a$ simultaneously approaching $\infty$, which would be needed for asymptotics
of $\tilde \Omega_n(u)$, one could use the parametrization $n=2\kappa^2\in\mathbb{N}, a=\lfloor2\kappa u\rfloor$, and consider
$$\frac{\omega_{\lfloor2\kappa u\rfloor,2\kappa^2}}\kappa=\frac{1}{2\pi \i}\oint_{C''}\frac1{\kappa}\frac{(-1)^{\lfloor2\kappa u\rfloor+1}\Gamma(-z)\Gamma(2z-1)}{\Gamma(z+\lfloor2\kappa u\rfloor)\Gamma(z-\lfloor2\kappa u\rfloor)\Gamma(z+1)}
\frac{(2\kappa^2)!\Gamma(-z)}{\Gamma(2\kappa^2+1-z)}dz,
$$
implied by \eqref{oman}, where $C''$ is a contour that encircles integers $1,\ldots,2\kappa^2$, but neither
any other integers, nor poles of the integrand. Outside $C''$, the integrand has poles at $\frac12$, at $0$, and at all negative
half-integers. For fixed $u$ it turns out that each residue contributes to the leading (constant) term of the asymptotics in the limit $\kappa\to\infty$, with the sum of those contributions converging, but for fixed $\kappa$ the sum of residues does not converge. Balancing those two limiting processes (taking more and more residues into account, letting
$\kappa\to\infty$) and at the same time bounding the integral over a sequence of correspondingly deformed contours appears to be intricate, so unfortunately we have not been able to prove $\tilde\Omega_n(u)\to\Omega(u)$ for $u\ne0$.

Using holonomicity of
$(\omega_{a,n})_{n\ge0}$ to generate many terms of that sequence for many values of $a$, we obtain the plots in Figures~\ref{fig:OmConv} and \ref{fig:logshefluc}.
The values for $n$ in Figure~\ref{fig:OmConv} and in the second plot in Figure~\ref{fig:logshefluc} have been chosen to satisfy $\sin(4\sqrt{n})\approx1$ in order to give maximal weight to the term $(-1)^a$ present in \eqref{omegarec} and thus ensure better comparability of the plots in the vicinity of $0$.
The value of $n$ in the first plot of Figure~\ref{fig:logshefluc} satisfies $\sin(4\sqrt{n})\approx0$.
We conclude this section with some (non-rigorous) observations based on these plots.

\begin{figure}
\centering
\includegraphics[scale=.6]{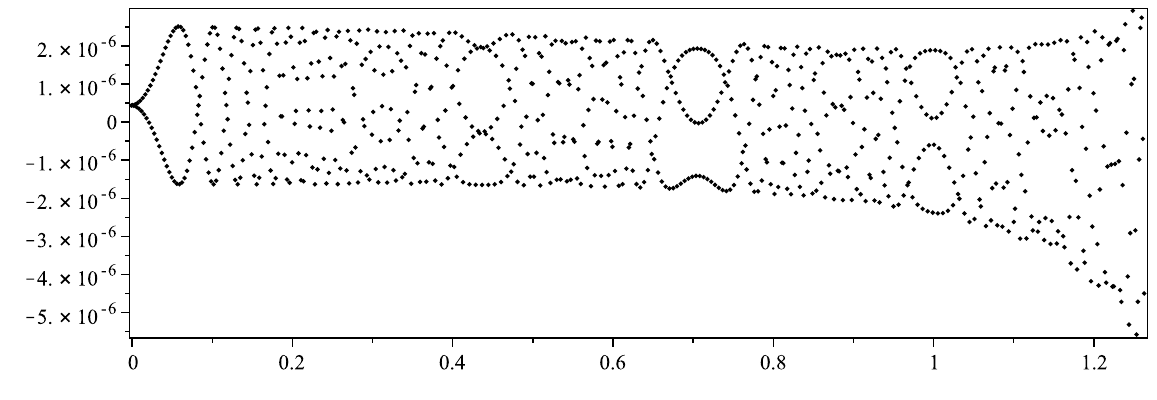}
\includegraphics[scale=.6]{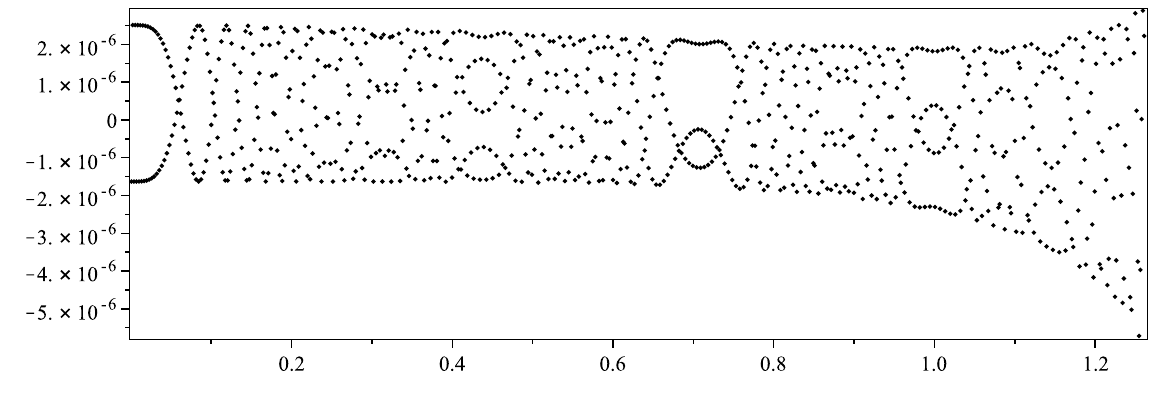}
\caption{Plots of $\tilde\Omega_n(u)-\Omega(u)$ for $n=200415$ (upper plot) and for $n=200767$ (lower plot), with $u$ varying in the discrete set $\{\frac{a}{\sqrt{2n}}:0\le a\le 800\}$.
}
\label{fig:logshefluc}
\end{figure}

a) Regarding convergence of the sequence $\big(\tilde\Omega_n(\cdot)\big)_{n\ge1}$ to $\Omega(\cdot)$, the plots in Figure~\ref{fig:OmConv} suggest that for any $\epsilon>0$ we have $\max\limits_{0\le u\le\sqrt{2}-\epsilon}|\tilde\Omega_n(u)-\Omega(u)|=\Oh\big(\frac1n\big)$.
Near $\sqrt{2}$ we only seem to have
$\max\limits_{|u-\sqrt{2}|\le\epsilon}|\tilde\Omega_n(u)-\Omega(u)|
=\Oh\big(\frac1{\sqrt{n}}\big)$, with this slower convergence in agreement with
the convergence rate stated in \cite[eq.\,(3.4)]{Bog07} for convergence of
$\tilde\Omega_n(\sqrt{2})\to\Omega(\sqrt{2})$ in probability, building upon results
from \cite{BDF99,BOO00,Joh01,Oko00} addressing the \emph{problem of the longest increasing subsequence}, see also \cite{Rom15}.

b) A common feature of all plots in Figures~\ref{fig:OmConv} and ~\ref{fig:logshefluc} is the presence of subintervals of ``smooth'' behavior surrounded by regions of more ``irregular'' behavior.
In order to enforce ``smooth'' dependence of $\omega_{a,n}$ on $a$, one would, in the light of \eqref{omegarec}, restrict to odd (or to even) $a$. However, this would only work for $0\le a\le \alpha_n$ with $\alpha_n=o(\sqrt{n})$. The location of the first ``peak'' to the right of $0$ seems to suggest, that $\alpha_n=\Theta\big( n^{\frac14}\big)$ may hold. For larger $a$ it is no longer useful to distinguish between even and odd, instead one should consider $\omega_{a,n}$ evaluated at $a$ belonging to other arithmetic progressions: Near $\frac a{\sqrt{2n}}=\sqrt{2}\cos\frac\pi3\approx0.707$ the way to go would be to consider $\big(\omega_{a+3k,n}\big)_{k}$, whereas near
$\frac a{\sqrt{2n}}=\sqrt{2}\cos\frac\pi4=1$ it would be $\big(\omega_{a+4k,n}\big)_{k}$. Every fifth term should be taken near
$\sqrt{2}\cos\frac\pi5\approx1.144$ and $\sqrt{2}\cos\frac{2\pi}5\approx0.437$.
We expect this pattern to continue, with regions of smoothness near $\sqrt{2}\cos\frac{\ell\pi}m$ for $1\le\ell<\frac m2$, and $\ell,m$ coprime. For larger $m$ these regions will become noticeable only if $n$ gets large enough, and those regions will shrink with $n$ further increasing, making room for yet other regions to pop up.

\section{A heuristic upper bound for the variance of the number of bumping steps in the RSK algorithm}
Let $L_n:=X_n+Y_n$, $\ell_n:=\EE (X_n+Y_n)$, and $v_n:=\Var (X_n+Y_n)$.
We now give a heuristic derivation of $\ell_n$, and an upper bound for $v_n$ based on \cite{Ker99}. Let
$\Omega(x)$ be the function defined in \eqref{e:lish}, describing
the limit shape of normalized Young diagrams with respect to Plancherel measure. Denote $s(x):=\frac1{\pi}\sqrt{2-x^2}$, the density
of the semicircle distribution with support $[-\sqrt{2},\sqrt{2}]$.
As shown in \cite{Ker99}, this is also the limiting density of the random abscissa of a newly inserted box into a scaled and rotated Young diagram that closely resembles the limit shape curve, when new insertions are made according to the Plancherel growth process, that ensures that at each stage of the process the Young diagram is distributed according to Plancherel measure, see also~\cite[sec.\,1.19]{Rom15}.
This leads to
$$\ell_n-\ell_{n-1}\sim\sqrt{2n}\int_{-\sqrt{2}}^{\sqrt{2}}\Omega(x)s(x)dx=\frac{128}{9\pi^2}\sqrt{n},$$
and thus $\ell_n\sim\frac{256}{27\pi^2}n^{\frac32}$. Moreover,
assuming independence of $L_{n-1}$ and
$L_n-L_{n-1}$,
$$v_n-v_{n-1}\sim 2n\int_{-\sqrt{2}}^{\sqrt{2}}\left(\Omega(x)-\frac{64\sqrt{2}}{9\pi^2}\right)^2s(x)dx=\frac{54\pi^4+2835\pi^2-32768}{162\pi^4}n,$$
and thus
$$v_n\sim\frac{54\pi^4+2835\pi^2-32768}{324\pi^4}n^2\approx0.01496867061\,n^2.$$
Numerically we have e.g. $\frac{v_{50}}{50^2}\approx0.01216526413$.
\begin{figure}
\centering
\includegraphics[scale=.55]{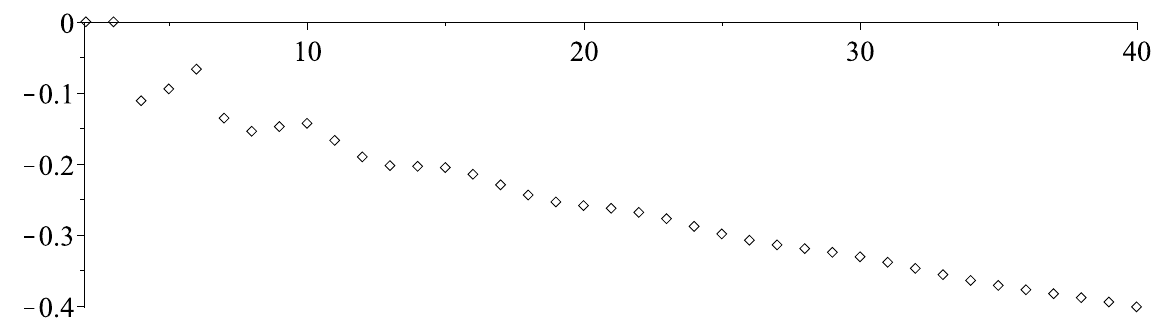}
\caption{A plot of $\Cov(L_n-L_{n-1},L_{n-1})$ for $n\in\{2,\ldots,40\}$.
}
\label{fig:cov40}
\end{figure}

Indeed, $L_{n-1}$ and $L_n-L_{n-1}$ seem to be negatively correlated.
The sequence of covariances $\big(\Cov(L_n-L_{n-1},L_{n-1})\big)_{n\ge2}$ starts
$(0,0,-\frac19,-\frac{17}{180},-\frac{1}{15},-\frac{61}{450},-\frac{863}{5600},\ldots)$, staying negative up to $n=40$ with roughly linear growth, see Figure~\ref{fig:cov40}.

So it seems that in the light of
Lemma~\ref{lem:bigOm} one can safely guess
that for the number $Y_n$ of bumping steps $\Var Y_n=\Theta(n^2)$ holds. It would be desirable to have a proof for that, and also know at least the leading asymptotic term of $\Var Y_n$.

\section{Conclusion}
In this paper we have obtained asymptotics of expectations of certain statistics of Plancherel distributed Young diagrams. That these statistics could be expressed in terms
of hook lengths and contents of the boxes of such diagram was essential, as it allowed us to invoke polynomiality results for Plancherel averages, leading to representations of expectations as binomial convolutions,
that make for easier asymptotic treatment. We hope that this approach will help to analyse further statistics of Plancherel distributed Young diagrams.
Now polynomiality results have also been found for measures different from Plancherel (such as the Jack deformation of Plancherel measure, see~\cite{Ols10}), or for subclasses of Plancherel distributed Young diagrams, such as strict partitions (see~\cite{Mat17,HaX19,MaS20}). In case that appropriate substitutes for \eqref{e:pan} or \eqref{e:fuj} are at hand, it is reasonable to believe that certain statistics
in these settings could also be analysed along the lines of this paper.

\section{Appendix}
\subsection{Proof of equation \eqref{e:id}:} We use $p(n,r)=\frac{(2r+1)!}{n}\binom{n+r}{2r+1}$ and rewrite \eqref{e:id} as
$$n^2=\sum_{r\ge0}\frac{(-1)^r}{1-2r}\binom{2r}{r}\binom{n+r}{2r+1}=:S_n$$
Denoting by $\Delta$ the forward difference operator, we will show that $\Delta^3S$ is the zero sequence, which together with $S_1=1,\Delta S_1=3,\Delta^2 S_1=2$ yields $S_n=n^2$ for $n\in\mathbb{N}$. Now
\begin{align*}\Delta^3S_n&=\sum_{r\ge1}\frac{(-1)^r}{1-2r}\binom{2r}{r}\binom{n+r}{2r-2}=\sum_{r\ge1}(-1)^{r+1}\frac{2}{r}\binom{2r-2}{r-1}\binom{n+r}{2r-2}\\
&=2\sum_{r\ge1}\frac{(-1)^{r+1}}{n+2}\binom{n+r}{r-1}\binom{n+2}{r}=\frac{2}{n+2}\sum_{r\ge0}(-1)^{r}\binom{n+r+1}{n+1}\binom{n+2}{r+1}=0,
\end{align*}
where the last equality follows from \cite[eq\,(5.24)]{GKP94}.\qed

\subsection{Proof of equation \eqref{e:kron}:} The equation is easily checked for $\ell> n$, since then also $r\ge n$ and thus $p(n,r)=0$. For $1\le\ell\le n$ we have
\begin{align*}\sum_{r=\ell-1}^{n-1}&\frac{2\ell^2(-1)^{r+\ell+1}}
{(r+\ell+1)!(r-\ell+1)!}\frac{(2r+1)!}{n}\binom{n+r}{2r+1}\\
&=\frac{2\ell^2(-1)^{\ell+1}}
{n(n+\ell)}\sum_{r=\ell-1}^{n-1}(-1)^r\binom{n+r}{n+\ell-1}\binom{n+\ell}{r+\ell{\color{black}+}1}\\
&=\frac{2\ell^2(-1)^{\ell+1}}{n(n+\ell)}(-1)^{n+1}\delta_{\ell,n}=\delta_{\ell,n},
\end{align*}
treating the case $\ell=n$ directly, and using \cite[eq\,(5.24)]{GKP94} again for $n>\ell$.\qed

\subsection{Proof of equation \eqref{e:kron2}:}
We use $q(n,r)=\frac{n(2r)!}{n+r}\binom{n+r}{2r}$ for $n>0$, and $q(0,0)=1$. The equation is easily checked for $n=0$, and for $\ell> n$, since then also $r> n$ and thus $q(n,r)=0$.
For $n>0,\,0\le\ell\le n$ we have
\begin{align*}\sum_{r=\ell}^{n}&\frac{2(-1)^{r+\ell}}
{(r+\ell)!(r-\ell)!}\frac{n(2r)!}{n+r}\binom{n+r}{2r}=(-1)^{\ell}\frac{2n}{n+\ell}
\sum_{r=\ell}^{n}(-1)^r\binom{n+r-1}{n+\ell-1}\binom{n+\ell}{r+\ell}\\
&=\frac{2n(-1)^{\ell}}{n+\ell}(-1)^{n}\delta_{\ell,n}=\delta_{\ell,n},
\end{align*}
treating the case $\ell=n$ directly, and using \cite[eq\,(5.24)]{GKP94} again for $n>\ell$.\qed

\subsection{Proof of equation \eqref{e:numberd}:}
Using $H_\ell-1=\sum_{k=2}^\ell\frac1k$, and interchanging summation, we obtain, using \cite[eq\,(5.16)]{GKP94} at several places,
\begin{align*}\sum_{\ell=2}^r\frac{(-1)^{\ell}\ell^2(H_\ell-1)}{(r+\ell)!(r-\ell)!}
&=\frac1{(2r)!}\sum_{k=2}^r\frac1k\sum_{\ell=k}^r(-1)^{\ell}[r^2-(r+\ell)(r-\ell)]\binom{2r}{r-\ell}\\
&=\frac1{(2r)!}\sum_{k=2}^r\frac1k\sum_{\ell=k}^r(-1)^{\ell}
\left[r^2\binom{2r}{r-\ell}-2r(2r-1)\binom{2r-2}{r-\ell-1}\right]\\
&=\frac1{(2r)!}\sum_{k=2}^r\frac1k(-1)^k
\left[r^2\binom{2r-1}{r-k}-2r(2r-1)\binom{2r-3}{r-k-1}\right]\\
&=\frac1{(2r)!}\sum_{k=2}^r\frac{(-1)^kr}{r-1}(k-1)
\binom{2r-1}{r-k}\\
&=\frac1{(2r)!}\sum_{k=2}^r\frac{(-1)^kr}{r-1}\left[(r-1)
\binom{2r-1}{r-k}-(2r-1)\binom{2r-2}{r-k-1}\right]\\
&=\frac1{(2r)!}\frac{r}{r-1}\left[(r-1)
\binom{2r-2}{r}-(2r-1)\binom{2r-3}{r}\right]\\
&=\frac1{(2r)!}\frac{r}{r-1}\frac{r}{r-2}\binom{2r-3}{r}
=\frac1{4(r-1)(2r-1)(r-1)!^2}.\qed
\end{align*}

\subsection{Saddle point evaluation of the integral $I_n:=\frac{1}{2\pi \i}\oint_{C'}f(z)\frac{n!\Gamma(-z)}{\Gamma(n+1-z)}dz$}
This integral appears in the proof of Theorem~\ref{t:Romik}, see section~2 for relevant notation.
Putting $n=m^2$, the integrand may be rewritten as
$$G(z):=\frac{\pi\Gamma(m^2+1)\Gamma^2(2z-1)}{\sin (\pi z)\Gamma(m^2+1-z)(2z-3)\Gamma^3(z+1)\Gamma^3(z)}.$$
Denoting by $\psi$ the digamma function, we have
\begin{align*}\frac{G'(z)}{G(z)}&=\psi(m^2+1-z)+4\psi(2z-1)-3\psi(z+1)-3\psi(z)
-\pi\cot\pi z-\frac2{2z-3}\\
&\sim\log\frac{(z-m^2-1)(2z-1)^4}{(z+1)^3z^3}-\frac1{z-\frac12}+\frac3{2z+2}
+\frac3{2z}-\frac1{z-\frac32},
\end{align*}
with error terms $\Oh(m^{-2})+\Oh(z^{-2})$, holding for $m\to+\infty,|z|\to\infty$, subject to $z=o(m^2)$ and $\delta<|\arg z|\le\pi-\delta$
for some $\delta>0$.

Two approximate saddle points of $G(z)$ are $\zeta:=6+4m\i$ and $\bar\zeta=6-4m\i$.
Indeed, $\frac{d}{dz}\log G(\zeta)=\frac{G'(\zeta)}{G(\zeta)}=\Oh(\frac1{m^2})$, and
$\frac{d^2}{dz^2}\log G(\zeta)=\frac\i{2m}+\Oh(\frac1{m^2})$, which suggests a
contour directed north-west in the point $\zeta$: Let $z=\zeta+e^{\i\frac{\pi}4}u$ and observe
$$G(z)=G(\zeta)\exp\left(\Oh(\tfrac{u}{m^2})+\tfrac{\i}{4m}(e^{\i\frac{\pi}4}u)^2+
\Oh(\tfrac{u^2}{m^2})\right)=G(\zeta)e^{-\frac{u^2}{4m}}
\left(1+\Oh(\tfrac{u+u^2}{m^2})\right).$$
Also note that a cumbersome evaluation results in
$$G(\zeta)=\frac{-\i e^{8\i m+8}}{2^{13}\pi m^4}\left(1+\Oh\left(\frac1m\right)\right).$$
Define the counter-clockwise oriented contour $C'$ as the polygon connecting the points
\begin{align*}z_0:=-\tfrac52+\eps,\quad& z_1:=\zeta-m-m\i,\quad z_2:=\zeta+m+m\i,\quad z_3:=n+1+m\i,\\
&z_6:=\bar\zeta-m+m\i,\quad z_5:=\bar\zeta+m-m\i,\quad z_4:=n+1-m\i,
\end{align*}
with $\eps>0$ small, and with segment $c_i$ connecting $z_i$ and $z_{i+1}$ for $0\le i\le5$, and $c_6$ connecting $z_6$ and $z_0$.
It turns out that the integrals along $c_0$ and $c_6$ are of order $\Oh(m^{-5+2\eps})$, and $c_i$, for $2\le i\le4$, make even smaller contributions.
Moreover, the combined contribution of $c_1$ and $c_5$ is $-2\sqrt{\frac{m}{\pi}}\Im\big(e^{\i\frac{\pi}4}G(\zeta)(1+\Oh(\frac1m))\big)$, which, up to error terms of order $\Oh(m^{-\frac92})$, simplifies to
$$-2\sqrt{\frac{m}{\pi}}\Im(e^{\i\frac{\pi}4}G(\zeta))
=\frac{e^8}{2^{12}\pi^\frac32m^\frac72}\Im(\i e^{\frac{\pi\i}4+8\i m})
=\frac{e^8}{2^{12}\pi^\frac32m^\frac72}\cos\left(\frac{\pi}4+8 m\right).$$

\subsection{Bounding the integral $J_n:=\frac{1}{2\pi \i}\oint_{C'}\Phi_n(z)\,dz$}
This integral appears in the proof of Theorem~\ref{t:Buf}, see section~3 for relevant notation. Let $C'$ be the boundary of the rectangle with corners $-\frac12\pm \i 4em$,
$n+\frac12\pm \i 4em$, and $m=\sqrt{n}$.

Observe $(-1)^{\ell}\ell^2\big(\log \ell-H_\ell+\gamma+\frac1{2\ell}-\frac1{12\ell^2}\big)=\Oh(\ell^{-2})$, and, abbreviating $\rho=r+\frac12$,
$$\left|\frac{\Gamma^2(r+1)}{\Gamma(r+\ell+1)\Gamma(r-\ell+1)}\right|=\left|\frac{r(r-1)\cdots(r-\ell+1)}{(r+1)\cdots(r+\ell)}\right|
=\left|\prod_{k=1}^\ell\frac{\rho-(k-\frac12)}{\rho+(k-\frac12)}\right|\le1
$$ for $\Re\rho\ge0$, i.e., for $\Re r\ge-\frac12$,
therefore $\Gamma^2(r+1)h(r)=\Oh(1)$ for $\Re r\ge-\frac12$, which leads to $\Gamma^2(z+1)g(z)=\Oh(1)$ for $\Re z\ge-\frac12$, $|z-w|\ge\frac12$ for $w\in\{0,\frac12,1\}$. Hence
$$\Phi_n(z)=\Oh\left(\frac{\Gamma(2z)\Gamma(2z-1)}{\Gamma^4(z+1)\Gamma(z)}\frac{n!\,\Gamma(-z)}{\Gamma(n+1-z)}\right)
=\Oh\left(\frac{2^{4z}}{z^4\Gamma^2(z+1)\sin\pi z}\frac{\Gamma(n+1)}{\Gamma(n+1-z)}\right),$$
by the reflection and duplication formulas.
For $z=-\frac12+\i t$, with $t=\Oh\big(\sqrt{n}\big)$, we have $|\Gamma^2(z+1)\sin\pi z|=\pi$ and $\left|\frac{\Gamma(n+1)}{\Gamma(n+1-z)}\right|=\Oh\big(n^{-\frac12}\big)$, yielding a contribution
$\Oh\big(n^{-\frac12}\big)$ from the integral over the left segment
of $C'$. The contributions from the two horizontal segments is
$\Oh\big(n^{-\frac32}\big)$, while the right segment makes an
exponentially small contribution.
All this can be seen from the estimates
$$\frac1{\Gamma^2(z+1)\sin\pi z}
=\Oh\left(\left(\frac{\sigma^2+t^2}{e^2}\right)^{-\sigma-\frac12}\right),$$ holding for $z=\sigma+\i t$ with $\sigma\ge-\frac12$ and
$|z-w|\ge\frac12$ for $w\in\mathbb{Z}$, and
$$\left|\frac{2^{4z}\Gamma(n+1)}{\Gamma(n+1-z)}\right|=\Oh\bigg((16n)^\sigma \exp\Big(\frac{2t^2}{n-\sigma+t}\Big)\bigg),$$
holding for $z=\sigma+\i t$ with $-\frac12\le\sigma\le n+\frac12$ and
$t=\Oh\big(\sqrt{n}\big)$.
\section*{Acknowledgements}
We would like to thank two anonymous referees, whose suggestions led to substantial improvements of the paper.


\begin{thebibliography}{99}
\bibitem{BDF99}
{J.~Baik, P.\,Deift, and K.\,Johansson}
\newblock  On the distribution of the length of the longest increasing
              subsequence of random permutations.
\newblock {\em J. Amer. Math. Soc.} 12:1119--1178, 1999.

\bibitem{Bog07}
{ L.\,V. Bogachev, and Z. Su.}
\newblock  Gaussian fluctuations of {Y}oung diagrams under the
              {P}lancherel measure.
\newblock {\em Proc. R. Soc. Lond. Ser. A Math. Phys. Eng. Sci.} 463:1069--1080, 2007.

\bibitem{BOO00}
{A.\,Borodin, A.\,Okounkov, and G.\,Olshanski.}
\newblock  Asymptotics of {P}lancherel measures for symmetric groups.
\newblock {\em J. Amer. Math. Soc.} 13:481--515, 2000.

\bibitem{Buf12}
{ A.\,I. Bufetov.}
\newblock  On the {V}ershik-{K}erov conjecture concerning the
              {S}hannon-{M}c{M}illan-{B}reiman theorem for the
              {P}lancherel
              family of measures on the space of {Y}oung diagrams.
\newblock {\em Geom. Funct. Anal.} 22:938--975, 2012.

\bibitem{CCS98}
{E.\,R. Canfield, S. Corteel, and C.\,D. Savage.},
\newblock  Durfee polynomials.
\newblock {\em    Electron. J. Combin.} 5, Research Paper 32, 1998.

\bibitem{Can05}
{E.\,R. Canfield.},
\newblock  From recursions to asymptotics: {D}urfee and dilogarithmic
              deductions.
\newblock {\em Adv. in Appl. Math.} 34:768--797, 2005.

\bibitem{FRT54}
{ J.\,S. Frame, G.\,de\,B. Robinson and R.\,M. Thrall.}
\newblock  The hook graphs of the symmetric groups.
\newblock {\em Canad. J. Math.} 6:316--324, 1954.

\bibitem{FKM08}
{ S. Fujii, H. Kanno and S. Moriyama.}
\newblock  Instanton calculus and chiral one-point functions in
              supersymmetric gauge theories.
\newblock {\em Adv. Theor. Math. Phys.} 12:1401--1428, 2008.

\bibitem{FlS95}
{ P. Flajolet, and R. Sedgewick.}
\newblock  Mellin transforms and asymptotics: finite differences and
              {R}ice's integrals.
\newblock {\em Theoret. Comput. Sci.} 144:101--124, 1995.

\bibitem{GKP94}
{ R.\,L. Graham, D.\,E. Knuth, and O. Patashnik.}
\newblock {\em Conrete Mathematics: A Foundation for Computer Science.}
\newblock Addison-Wesley, \mbox{1994.}

\bibitem{HaX19}
{G.-N.\,Han, and H.\,Xiong},
\newblock Polynomiality of {P}lancherel averages of hook-content
              summations for strict, doubled distinct and self-conjugate
              partitions.
\newblock {\em J. Combin. Theory Ser. A}, 168:50--83, 2019.

\bibitem{Joh01}
{K.\,Johansson.}
\newblock  Discrete orthogonal polynomial ensembles and the {P}lancherel
              measure.
\newblock {\em Ann. of Math.} 153:259--296, 2001.

\bibitem{Ker99}
{S. Kerov.}
\newblock A differential model for the growth of Young diagrams.
\newblock in {\em Proceedings of the St. Petersburg
Mathematical Society, Vol. IV, Amer. Math. Soc. Transl. Ser. 2}, vol. 188, pp. 111--130. Amer. Math. Soc., Providence, RI, 1999.

\bibitem{LoS77}
{B.\,F. Logan, and L.\,A. Shepp.}
\newblock A variational problem for random Young tableaux.
\newblock {\em Adv. Math.} 26:206--222, 1977.

\bibitem{Luk69}
{ Y.\,L. Luke.}
\newblock {\em The special functions and their approximations, {V}ol. {I}.}
\newblock Academic Press, \mbox{1969.}

\bibitem{Mat17}
{ S.\,Matsumoto},
\newblock Polynomiality of shifted {P}lancherel averages and content
              evaluations.
\newblock {\em Ann. Math. Blaise Pascal}, 24:55--82, 2017.

\bibitem{MaS20}
{ S.\,Matsumoto, and P.\,\'{S}niady},
\newblock Random strict partitions and random shifted tableaux.
\newblock {\em Sel. Math. New Ser.} 26, 10, 2020.

\bibitem{Mut02}
{ L.\,R. Mutafchiev.}
\newblock On the size of the {D}urfee square of a random integer
              partition.
\newblock {\em J. Comput. Appl. Math.} 142:173--184, 2002.

\bibitem{Oko00}
{A.\,Okounkov.}
\newblock Random matrices and random permutations.
\newblock {\em Internat. Math. Res. Notices}, 20:1043--1095, 2000.

\bibitem{Ols10}
{ G. Olshanski.}
\newblock Plancherel averages: remarks on a paper by {S}tanley.
\newblock {\em Electron. J. Combin.} 17, Research Paper 43, 2010.

\bibitem{Pan12}
{ G. Panova.}
\newblock Polynomiality of some hook-length statistics.
\newblock {\em Ramanujan J.} 27:349--356, 2012.

\bibitem{Rom05}
{ D. Romik.}
\newblock The number of steps in the Robinson-Schensted algorithm.
\newblock {\em Funct. Anal. Appl.} 39:152--155, 2005.

\bibitem{Rom15}
{ D. Romik.}
\newblock {\em The surprising mathematics of longest increasing subsequences.}
\newblock Cambridge University Press, New York, \mbox{2015.}

\bibitem{Sag01}
{ B.\,E. Sagan.}
\newblock {\em The symmetric group.}
\newblock Springer, New York, \mbox{2001.}

\bibitem{SZ94}
{B. Salvy and P. Zimmermann.}
\newblock  Gfun: a Maple package for the manipulation of generating and holonomic functions in one variable.
{\em ACM T. Math. Software}, 20(2):163--177, 1994.

\bibitem{Stan10}
{ R.\,P. Stanley.}
\newblock Some combinatorial properties of hook lengths, contents, and parts
of partitions.
\newblock {\em Ramanujan J.} 23:91--105, 2010.

\bibitem{VeKe77}
{ A.\,M. Vershik, and S.\,V. Kerov.}
\newblock Asymptotics of the Plancherel measure
of the symmetric group and the limiting shape of Young tableaux.
\newblock {\em Soviet Math. Dokl.} 18:527--531, 1977.

\bibitem{VeKe85}
{ A.\,M. Vershik, and S.\,V. Kerov.}
\newblock Asymptotic behavior of the maximum and generic dimensions of
irreducible representations of the symmetric group.
\newblock {\em Funktsional. Anal. i Prilozhen.} 19:25--36, 96, 1985.

\bibitem{VePa10}
{ A. Vershik, and D. Pavlov.}
\newblock Numerical experiments in problems of asymptotic representation
theory
\newblock {\em J. Math. Sci.} 168:351--361, 2010.

\end{thebibliography}
\end{document}